\documentclass[12pt,reqno]{amsart}
\usepackage[usenames]{color}
\usepackage{amsmath,pdfsync,verbatim,graphicx,epstopdf,enumerate}
\usepackage[colorlinks=true]{hyperref}
\hypersetup{urlcolor=blue, citecolor=red}
\numberwithin{equation}{section}

\newcommand{\I}{{i}}
\newcommand{\D}{\mathrm{d}}

\newcommand{\wh}{\widehat}

\newcommand{\lb}{\left(}
\newcommand{\Lm}{\left\lvert}
\newcommand{\Rm}{\right\rvert}
\newcommand{\vp}{\varphi}
\newcommand{\ve}{\varepsilon}

\newcommand{\rb}{\right)}

\newcommand{\wt}{\widetilde}

\newcommand{\Dc}{\mathcal{D}}

\newcommand{\Lc}{\mathcal{L}}
\newcommand{\Nc}{\mathcal{N}}
\newcommand{\Oc}{\mathcal{O}}

\newcommand{\Cb}{\mathbb{C}}

\newcommand{\Rb}{\mathbb{R}}

\newcommand{\Beq}{\begin{equation}}
\newcommand{\Eeq}{\end{equation}}
\newcommand{\beq}{\begin{equation*}}
\newcommand{\eeq}{\end{equation*}}
\newcommand{\bal}{\begin{align}}
\newcommand{\eal}{\end{align}}
\renewcommand{\O}{\Omega}

\newcommand{\LM}{\left\lVert}
\newcommand{\RM}{\right\rVert}

\newcommand{\A}{\alpha}
\newcommand{\B}{\beta}
\newcommand{\bp}{\begin{prob}}
\newcommand{\ep}{\end{prob}}
\newcommand{\bpr}{\begin{proof}}
\newcommand{\epr}{\end{proof}}
\renewcommand{\o}{\omega}

\newcommand{\bel}[1]{\begin{equation}\label{#1}}
\newcommand{\ee}{\end{equation}}

\newcommand{\ssubset}{\subset\joinrel\subset}

\newtheorem{theorem}{Theorem}[section]
\newtheorem{corollary}[theorem]{Corollary}

\newtheorem{proposition}[theorem]{Proposition}

\theoremstyle{definition}
\newtheorem{definition}[theorem]{Definition}
\newtheorem{remark}[theorem]{Remark}

\title[Inverse Problem For The Polyharmonic Operator]{Inverse Boundary Value Problem of Determining Up to a Second Order Tensor Appear in the Lower Order Perturbation of a Polyharmonic Operator}
\author[Bhattacharyya and Ghosh]{Sombuddha Bhattacharyya$^\ast$ and Tuhin Ghosh$^{\ast\ast}$}
\address {$^{\ast}$Centre for Applicable Mathematics, Tata Institute of Fundamental Research,
\newline
\indent\: 
E-mail:{\tt\  arkatifr@gmail.com}}
\address{$^{\ast\ast}$ Institute for Advanced Study,  Hong Kong University of Science and Technology.
\newline
\indent\: \
E-mail:{\tt \  iasghosh@ust.hk}}

\begin{document}

\begin{abstract}

We consider the following perturbed polyharmonic operator $\Lc(x,D)$ of order $2m$ defined in a bounded domain $\Omega \subset \mathbb{R}^n, n\geq 3$ with smooth boundary, as
\begin{equation*}
\Lc(x,D) \equiv (-\Delta)^m + \sum_{j,k=1}^{n}A_{jk} D_{j}D_{k} + \sum_{j=1}^{n}B_{j} D_{j} + q(x),
\end{equation*} 
where $A$ is a symmetric $2$-tensor field, $B$ and $q$ are vector field and scalar potential respectively. 
We show that the coefficients $A=[A_{jk}]$, $B=(B_j)$ and $q$ can be recovered from the associated Dirichlet-to-Neumann data on the boundary. Note that, this result shows an example of determining higher order ($2$nd order) symmetric tensor field in the class of inverse boundary value problem. 
\end{abstract}

\maketitle
\section{Introduction and statement of the main results}
Let $\Omega\subset\Rb^{n},$ $n\geq 3$ be a bounded domain with smooth connected boundary. 
Let us consider the following perturbed polyharmonic operator $\Lc(x,D)$ of order $2m$, with perturbations up to second order, of the following form: 
\begin{equation}\label{operator}
\Lc(x,D)\equiv
(-\Delta)^m + \sum_{j,k=1}^{n} A_{jk}(x) D_{j}D_{k} + \sum_{j=1}^{n} B_{j}(x) D_{j} + q(x); \quad m\geq 2,
\end{equation}
where $D_j=\frac{1}{i}\partial_{x_j}$, $A \in W^{3,\infty}(\Omega, \mathbb{C}^{n^2})$,  $B \in W^{2,\infty}(\Omega, \mathbb{C}^{n})$ and $q \in L^{\infty}(\Omega, \mathbb{C})$. For simplicity, we call 
\[A_\alpha :\Omega \mapsto \mathbb{C}^{n^{|\alpha|}}, \mbox{ as } A_\alpha= A\mbox{ whenever }|\alpha|=2, \mbox{ and } \ A_\alpha = B \mbox{ whenever }|\alpha|=1.\]  
Consider the domain of this operator to be 
\[
\Dc(\Lc(x,D)) = \Big{\{}u\in H^{2m}(\Omega) \mid u|_{\partial\Omega}=(-\Delta)u|_{\partial\Omega}=\cdots=(-\Delta)^{m-1}u|_{\partial\Omega} =0\Big{\}}.
\]
The operator $\Lc(x,D)$ with the domain $\Dc(\Lc(x,D))$ is an unbounded closed operator on $L^{2}(\O)$ with a purely discrete spectrum \cite{Gerd_Grubb}.
We make the assumption that $0$ is not an eigenvalue of the operator $\Lc(x,D) : \Dc(\Lc(x,D)) \to L^2(\O)$. Let us denote
\[ \gamma u = \Big{(} u |_{\partial\Omega},\cdots,(-\Delta)^k u |_{\partial\Omega},\cdots, (-\Delta)^{m-1}u |_{\partial\Omega}\Big{)},\]
then for any $ f = (f_0,f_1,..,f_{m-1}) \in \prod_{i=0}^{m-1} H^{2m-2i-\frac{1}{2}}(\partial\Omega)$, the boundary value problem, 
\begin{equation}\begin{aligned}\label{problem}
\Lc(x,D)u &= 0\quad \mbox{ in }\Omega \\
\gamma u&= f \quad \mbox{ on }\partial\Omega
\end{aligned}
\end{equation}
has a unique solution $u\in H^{2m}(\Omega)$.

Let us define the corresponding Neumann trace operator $\gamma^{\#} $ by
\[
\gamma^{\#} u = \Big{(}\partial_{\nu} u |_{\partial\Omega},\cdots,\partial_{\nu}(-\Delta)^k u |_{\partial\Omega},\cdots,\partial_{\nu} (-\Delta)^{m-1}u |_{\partial\Omega}\Big{)}
\]
where $\nu$ is the outer unit normal to the boundary $\partial\Omega$,
and the corresponding Dirichlet-to-Neumann map (D-N map) is given by,
\begin{equation}\label{dnmap}\begin{aligned}
 &\Nc : \prod_{i=0}^{m-1} H^{2m-2i-\frac{1}{2}}(\partial\Omega) \to \prod_{i=0}^{m-1} H^{2m-2i-\frac{3}{2}}(\partial\Omega)\\
 &\Nc(f) = \gamma^{\#} u = \Big{(}\partial_{\nu} u |_{\partial\Omega},\cdots,\partial_{\nu}(-\Delta)^k u |_{\partial\Omega},\cdots,\partial_{\nu} (-\Delta)^{m-1}u |_{\partial\Omega}\Big{)}
 \end{aligned}
 \end{equation}
where $u\in H^{2m}(\Omega)$ is the solution of \eqref{problem}. We also define the Cauchy data set as the graph of the D-N map
\begin{equation}\label{qw19} \mathcal{C}^N= \Big{(}u |_{\partial\Omega},\cdots, (-\Delta)^{m-1}u |_{\partial\Omega}, \partial_{\nu} u |_{\partial\Omega},\cdots,\partial_{\nu} (-\Delta)^{m-1}u |_{\partial\Omega} \Big{)}
\end{equation}
Equations of this kind has its own resemblance in the study of continuum mechanics of buckling problems \cite{Ashbaugh_Buckling} and in the related field of elasticity \cite{Polyharmonic_Book}.

Inverse problems concerning higher order operators is an active field of research in recent days. It addresses the recovery of the coefficient from measurements taken on full or on a part of the boundary \cite{IKE,ISA,KRU2,KRU1,VT,TS,AYY,Bhattacharyya2017}, lower regularity of the coefficients \cite{KU,AY,KIY}, stability issues \cite{VA,APH} from the boundary Dirichlet-Neumann data. In this paper we are interested in the recovery of the coefficients in \eqref{operator} from the boundary Dirichlet-Neumann data. To be specific, we are interested in the recovery of the second order tensor or matrix $A_\A$ for $|\A|=2$ associated with the second order perturbed term in the operator $\mathcal{L}(x,D)$. Recovery of lower order perturbations mainly zeroth order and first order, of biharmonic and polyharmonic operators has been considered in prior works
\cite{IKE, ISA, KRU1,KRU2,YANG,SEROV, SEROV2}. In \cite{IKE, ISA}, recovery of the zeroth order perturbations of the biharmonic operator was studied, and recently in \cite{KRU1, KRU2} recovery of the first and zeroth order perturbations of the biharmonic operator with partial boundary Neumann data and of the polyharmonic operator with full boundary Neumann data were considered. Similar work on non polyharmonic higher order elliptic operator $(\sum_{i=1}^n D^{2m}_{x_i},\, m\geq 2 )$ has been also considered in \cite{TS}. A natural question to ask is whether higher order perturbations of the polyharmonic operator can be recovered from boundary Dirichlet-Neumann map. Such attempted was initiated in \cite{VT}, where the authors had considered the higher order perturbations (upto $m^{\textrm{th}}$ order) of polyharmonic operator $(-\Delta)^m$, but the coefficients attached to those lower order perturbed terms were restricted among the class of first order and zeroth order tensors accordingly to odd and even order perturbations. So, the question becomes more interesting when one involves higher order tensor with the higher order perturbations.  In this paper, we show that if we consider a perturbation of the polyharmonic operator $(-\Delta)^{m}$ of the form \eqref{operator}, with $m>2$ then all the coefficients $A,B,q$ can be recovered from the boundary Dirichlet-Neumann data. For $m=2$, we recover $A$ as an isotropic matrix along with $B$ and $q$, this can be seen in \cite{VT}. In other words, this article provides an example of recovery of matrix in this class of inverse boundary value problems.

For the biharmonic operator that for being $m=2$, one may have gauge invariance in the recovery of the matrix coefficient. For instance, let us consider the following fourth order generalization of the inverse problem associated with the magnetic Schr\"{o}dinger operator \cite{SUN}, as   $u\in H^4(\Omega)$ solves  
\begin{equation}\label{g-eq}
 \sum_{k,l=1}^n(D_{k}D_{l} + A_{kl})^2u + qu = 0 \mbox{ in }\Omega,
\end{equation}
with the given Navier boundary condition $(u|_{\partial\Omega}, (-\Delta u)|_{\partial\Omega})\in H^{\frac{7}{2}}(\partial\Omega)\times H^{\frac{5}{2}}(\partial\Omega)$.\\
Note that, \eqref{g-eq} can be written in the form of \eqref{operator} as 
\[ (-\Delta)^2u + \sum_{k,l}A^{\prime}_{kl}D_{k}D_{l}u+\sum_jB^\prime_j D_ju + q^\prime u =0\] 
for some $A^\prime, B^\prime$ and $q^\prime$ as a function of $A,q$.

Then for any $\varphi\in C^4(\Omega)$ with \textit{supp of} $\varphi \ssubset \Omega$, we observe  $ue^\varphi\in H^4(\Omega)$ solves 
\begin{equation}
\sum_{k,l=1}^n\lb D_{k}D_{l} + A_{kl}+\frac{\partial^2\varphi}{\partial x_k\partial x_l}+\frac{\partial\varphi}{\partial x_k}\frac{\partial\varphi}{\partial x_l}\rb^2u + qu = 0 \mbox{ in }\Omega.
\end{equation}
Since $\mbox{supp}\, (\varphi) \subset \Omega$, then $(ue^\varphi, (-\Delta(ue^\varphi)))$ and the normal derivative $(\partial_\nu (ue^\varphi), \partial_\nu(-\Delta(ue^\varphi)))$ carries the same boundary value as $(u,(-\Delta)u)$ does. Therefore the DN map \eqref{dnmap} remains unchanged, i.e. $\mathcal{N}_{A,q}= \mathcal{N}_{\widetilde{A},q}$, where $\widetilde{A}$ equals to its gauge invariance 
\begin{equation}\label{gauge}
\widetilde{A}_{kl}= A_{kl}+\frac{\partial^2\varphi}{\partial x_k\partial x_l}+\frac{\partial\varphi}{\partial x_k}\frac{\partial\varphi}{\partial x_l}. 
\end{equation}

This stands as an obstacle to the uniqueness of $A$. However, if we consider $A_{ij}$ to be an isotropic matrix, i.e. $A_{ij}(x)=a(x)\delta_{ij}$ for some scalar function $a$ and $\delta_{jk}$ denotes the Kronecker delta functions, then full recovery of that isotropic matrix is possible and was first explored in \cite{VT}. 
We also note that the gauge invariance $(\widetilde{A}-A)$ in \eqref{gauge} is not governed by the $2$ form $d_sV$ where $V$ is a $1$-tensor field and $d_sV$ is the symmetrized inner derivative defined as \[ (d_sV)_{jk} = \frac{1}{2}\left( \frac{\partial }{\partial x_j}V_k + \frac{\partial }{\partial x_k}V_j \right). \]
This follows from the fact that $(\widetilde{A}-A)$ does not belong to the kernel of the corresponding \textit{Saint Venant operator} $W$ (see \cite[Theorem 2.2.1]{VS}) defined as 
\[W_{ijkl}(F) = \frac{\partial^2 F_{ij}}{\partial x_k \partial x_l} + 
\frac{\partial^2 F_{kl}}{\partial x_i \partial x_j} - \frac{\partial^2 F_{il}}{\partial x_j \partial x_k} -\frac{\partial^2 F_{jk}}{\partial x_i \partial x_l}, \]
where $F$ is a symmetric $2$ tensor field.
In our case one can get $\phi \in C^4_c(\mathbb{R}^n)$ such that
\begin{equation}
W(\widetilde{A}-A) = W(\frac{\partial^2\varphi}{\partial x_k\partial x_l}+\frac{\partial\varphi}{\partial x_k}\frac{\partial\varphi}{\partial x_l}) \neq 0.
\end{equation}

Whereas, in the case of the magnetic Schr\"{o}dinger operator $(D_k +A_k)^2 +q$, the gauge invariance of the vector field $(\widetilde{A}-A)$ with the same boundary data is governed by the $1$ form $dV$ where $V$ is a scalar field, see \cite{SUN, SALOT}.

This article is a natural extension of the results obtained in \cite{VT,KRU1,KRU2} which considered the recovery of only zeroth and first order tensor from the lower order perturbations. 
To the best of the authors knowledge, inverse problems involving upto higher order perturbations of polyharmonic operator which are essentially determined by higher order tensors has not been investigated in previous studies. The novelty of this work provides a fresh look in order to determine the tensors of order more than one, 	which we don't see often. 

Here we take the opportunity to mention that inverse problems with boundary information arise naturally in several imaging applications including seismic and medical imaging, electrical impedance tomography to name a few. The techniques we rely on to prove our main result are based on the pioneering works done for inverse boundary value problems involving Schr\"{o}dinger operators \cite{Calderon_Paper,SU,BUK,KEN,DOS}.

Now we state the main result of this article. 

\begin{theorem}\label{mainresult}
Let $\Omega\subset \mathbb{R}^n, n\geq 3 $ be a bounded domain with smooth connected boundary. Let $\Lc(x,D)$ and $\widetilde{\Lc}(x,D)$ be two operators defined as in \eqref{operator} having $m\geq 2$ with the coefficients $A_\A, \widetilde{A}_\A \in W^{1+\lvert\A\rvert,\infty}(\mathbb{R}^n,\mathbb{C}^{n^{|\alpha|}}) \cap \mathcal{E}^{\prime}(\overline{\Omega})$, $\lvert \alpha \rvert =1,2$ and $q,\widetilde{q} \in L^{\infty}(\Omega,\mathbb{C})$. We assume that for $m=2$, $A_\alpha, \widetilde{A}_\A$ with $|\alpha|=2$ are isotropic matrix. Let us assume that $0$ is not an eigenvalue of the operators $\Lc$, $\widetilde{\Lc}$ and $\Nc$, $\widetilde{\Nc}$ are the corresponding Dirichlet-to-Neumann maps respectively. If 
\begin{equation*}\label{Neumann Data}
\begin{aligned}
&\Nc(f)|_{\partial\Omega} = \widetilde{\Nc}(f)|_{\partial\Omega} \quad \mbox{ for all } f \in \prod_{i=0}^{m-1} H^{2m-2i-\frac{1}{2}}(\partial\Omega),
\end{aligned}
\end{equation*}
then 
\[
A_{\A} = \widetilde{A}_{\A},\quad\mbox{for } \lvert \alpha \rvert = 1,2; \quad\mbox{and}\quad q = \widetilde{q}, \quad \mbox{in }\Omega.
\]
\end{theorem}

\begin{remark}
The result for $m=2$ in this Theorem \ref{mainresult} with the assumption the second order perturbations are governed by isotropic matrices, follows from \cite[Theorem 1.1]{VT}.    
\end{remark}

Finally let us consider the same problem \eqref{problem} given with the 
Dirichlet boundary conditions instead of Navier boundary conditions. 
Let us denote 
\[ \gamma_Du =  \Big{(} u |_{\partial\Omega},\partial_\nu u|_{\partial\Omega},\cdots,\partial_\nu^k u |_{\partial\Omega},\cdots, \partial_\nu^{m-1}u |_{\partial\Omega}\Big{)}\]
Then for $f= (f_0,f_1,\cdots,f_{m-1}) \in \prod_{i=0}^{m-1} H^{2m-i-\frac{1}{2}}(\partial\Omega)$ we consider the boundary value problem 
\begin{equation}\begin{aligned}\label{dirichlet}
\mathcal{L}(x,D)u &= 0 \quad\mbox{ in }\Omega \\
\gamma_Du &= f \quad\mbox{ on }\partial\Omega. 
\end{aligned}
\end{equation}
The corresponding Neumann trace is 
\[\gamma_D^{\#} = \Big{(}\partial_{\nu}^{m}u|_{\partial\Omega},\cdots, \partial_{\nu}^{2m-1}u|_{\partial\Omega}\Big{)} \in\prod_{i=m}^{2m-1} H^{2m-i-\frac{1}{2}}(\partial\Omega)\]
where $u\in H^{2m}(\Omega)$ is the solution to the Dirichlet problem \eqref{dirichlet}. See \cite{AGM,Gerd_Grubb} for the wellposedness of the forward problem \eqref{dirichlet}.
We introduce the set of Cauchy data for the operator $\mathcal{L}(x,D)$ with the Dirichlet 
boundary condition by
\begin{equation*}
\mathcal{C}^D = \Big{(}u|_{\partial\Omega},\ldots,\partial^{m-1}_{\nu}u|_{\partial\Omega},\partial^m_{\nu}u|_{\partial\Omega},\ldots,\partial^{2m-2}_{\nu}u|_{\partial\Omega},\partial_{\nu}^{2m-1}u|_{\partial\Omega}\Big{)}
\end{equation*}
where $u\in H^{2m}(\Omega)$ solving \eqref{dirichlet}. We have this following result: 
\begin{corollary}
We assume $m\geq 2$ and  $\A=1,2$, $A_{\A},\widetilde{A}_{\A}$ and $q,\widetilde{q}$
satisfy the same conditions as in Theorem \ref{mainresult}. 
Then $\mathcal{C}^D= \widetilde{\mathcal{C}}^D$ implies that for each $\lvert\A\rvert =1,2$, $A_{\A} = {\wt{A}_\A}$ and 
$q = \widetilde{q}$ in $\Omega$. 
Proceeding in a similar way as of Theorem \ref{mainresult}, in this case we end up with an integral identity same as \eqref{integralidentityE}. 
Then following the same analysis one can show uniqueness of the lower order perturbations in $\Omega$.
\end{corollary}
\noindent
This article is organized as follows. 
Section \ref{Carleman Estimates} is devoted to the interior Carleman estimate, which is used to prove the existence of a complex geometric optics (C.G.O.) type solution for the equation \eqref{problem}. 
Finally, in Section \ref{determination} we derive an integral identity involving the perturbations and then prove the main result using Fourier analysis techniques.

\section{Carleman estimate and C.G.O. solutions}\label{Carleman Estimates}

In this section we derive Complex Geometric Optics (CGO) type solutions for the operator $\Lc$ and its formal $L^{2}$ adjoint $\Lc^{*}$ based on Carleman estimate. 

\subsection{Interior Carleman estimates} 
Let us assume that   
\begin{equation}\label{generaloperator}
 \Lc(x,D) =(-\Delta)^m + \sum_{ \lvert \A \rvert = 1}^{2} A_{\A}(x) D^{\A} + q
\end{equation}
where $\A$ is a multi-index, $A_{\A}\in W^{1+\lvert\A\rvert,\infty}(\Omega,\mathbb{C}^{n^{|\alpha|}})$ and $q \in L^{\infty}(\Omega)$.
Note that $\Lc^{*}$ the the $L^2$ adjoint of $\Lc$, also has the same form as that of $\Lc$ with different $A_\alpha$, and $q$. 
We will derive an interior Carleman estimate for the conjugated semiclassical version of the operator $\Lc$ as well as its adjoint operator. 

First we consider the principal part of the conjugated semiclassical version of the perturbed operator $\Lc(x,hD)$, which is given as $(-\Delta)^m$.
Then by adding the lower order terms to it finally we derive the required Carleman estimate for the conjugated semiclassical version of the operator $\Lc(x,D)$.

We start by recalling the definition of a limiting Carleman weight for the semiclassical Laplacian $-h^2\Delta$. 
Let $\Omega\subset \mathbb{R}^n$, $n\geq 3$ be a bounded domain with smooth boundary. 
Let $\wt{\O}$ be an open set in $\mathbb{R}^n$ such that 
$\Omega\ssubset \wt{\O}$ and $\vp\in C^{\infty}(\wt{\O},\mathbb{R})$.
Consider the conjugated operator 
$P_{0,\vp} = e^{\frac{\vp}{h}}(-h^2\Delta)e^{-\frac{\vp}{h}}$ 
with its semiclassical symbol $p_{0,\vp}(x,\xi)$. 

\begin{definition}[\cite{KEN}]
We say that $\vp$ is a limiting Carleman weight for $(-h^2\Delta)$
in $\wt{\O}$ if $\nabla\vp \neq 0 $ in $\wt{\O}$ and the Poisson 
bracket of $\mathrm{\mathrm{Re}}(p_{0,\vp})$ and  $\mathrm{Im}(p_{0,\vp})$ satisfying
\[\Big{\{}\mathrm{Re}(p_{0,\vp}), \mathrm{Im}(p_{0,\vp})\Big{\}}(x,\xi)=0 \quad \mbox{ whenever } p_{0,\vp}(x,\xi)=0 \mbox{ for } (x,\xi)\in (\overline{\Omega}\times\mathbb{R}^n).
\]
\end{definition}
An example of such $\vp$ is the linear weight defined as 
$\vp(x) = \A\cdot x$, where $\A\in\mathbb{R}^n$ with $|\A|=1$ or 
logarithmic weights $\vp(x)= \log|x-x_0|$ with $x_0 \notin \wt{\O}$.
Throughout this article we consider the limiting Carleman weight to be $\vp(x)=(\A\cdot x)$ where $\A\in\mathbb{R}^n$ with $|\A|=1$.

As the principal symbol of the semiclassical conjugated 
operator $e^{\frac{\vp}{h}}(-h^2\Delta)^{m}e^{-\frac{\vp}{h}}$ 
is given by $p_{0,\vp}^m$ which is not of principal type, the idea of
Carleman weight for polyharmonic operator is irrelevant. 
One always has the Poisson bracket of $\mathrm{Re}(p^m_{\vp})$ and $\mathrm{Im}(p^m_{\vp})$ 
is zero when $p^m_{\vp}(x,\xi)=0$, $(x,\xi)\in (\overline{\Omega}\times\mathbb{R}^n)$.
In order to get the Carleman estimate for the polyharmonic operator we iterate the Carleman estimate for the semiclassical Laplacian.

We use the semiclassical Sobolev spaces $H^s_{\mathrm{scl}}(\mathbb{R}^n)$ with $s\in\mathbb{R}$ equipped with the norm 
\begin{align*}
\lVert u\rVert_{H^s_{\mathrm{scl}}(\mathbb{R}^n)} = \lVert {\langle hD\rangle}^su\rVert_{L^2(\mathbb{R}^n)}
\end{align*} where 
$\langle \xi \rangle = (1+|\xi|^2)^{\frac{1}{2}}$. 

With these notations we now prove the following proposition.
\begin{proposition}\label{Prop: Interior Carleman Estimate}
	Let for each $\A$ with $\lvert \A\rvert = 1,2$, $A_{\A} \in W^{1+\lvert\A\rvert,\infty}(\Omega,\mathbb{C}^{n^{|\alpha|}})$, $q\in L^{\infty}(\O,\Cb)$
	and $\vp$ be a limiting Carleman weight for the semiclassical Laplacian on $\wt{\O}$. Then for $0< h\ll 1 $ and $-2m\leq s\leq 0$, we have
	\begin{equation}\label{carlemanestimate3}
		h^{m}\lVert u\rVert_{H^{s+2m}_{\mathrm{scl}}} \leq C \lVert h^{2m}e^{\frac{\vp}{h}}\Lc(x,D)e^{-\frac{\vp}{h}}u\rVert_{H^{s}_{\mathrm{scl}}},\mbox{ for all } u\in C^{\infty}_0(\Omega).
	\end{equation}
	the constant $C = C_{s,\O, A_{\A},q}$ is independent of $h$. 
\end{proposition}
\begin{proof}
Let us consider the convexified Carleman weight as in \cite{KEN} defined as 
		\[
		\vp_{\ve}=\vp+ \frac{h}{2\ve} \vp^{2} \mbox{ on } \wt{\O}.
		\]
We begin with the Carleman estimate for the semiclassical Laplacian with a gain of two derivatives proved in \cite{KEN}:
	\begin{equation}\label{carlemanesti}
	\frac{h}{\sqrt{\epsilon}}\lVert u\rVert_{H^{s+2}_{\mathrm{scl}}} \leq C \LM e^{\frac{\vp_{\ve}}{h}}(-h^2\Delta) e^{-\frac{\vp_{\ve}}{h}}u\RM_{H^s_{\mathrm{scl}}}, \mbox{ for all } u\in C^{\infty}_0(\Omega).
	\end{equation}
Let $-2m\leq s\leq 0$, then iterating the above estimate \eqref{carlemanesti} $m$ times, we get the following estimate: 
\begin{equation*}
	\left(\frac{h}{\sqrt{\epsilon}}\right)^m\lVert u\rVert_{H^{s+2m}_{\mathrm{scl}}} \leq C \LM e^{\frac{\vp_{\ve}}{h}}(-h^2\Delta)^m e^{-\frac{\vp_{\ve}}{h}}u\RM_{H^s_{\mathrm{scl}}}, \mbox{ for all } u\in C^{\infty}_0(\Omega).
	\end{equation*}
Let us consider $s \in [-2m,0]$ and we get
\begin{equation}\label{carlemanestimate22}
	\left(\frac{h}{\sqrt{\epsilon}}\right)^m\lVert u\rVert_{H^{s+2m}_{\mathrm{scl}}} \leq C \LM e^{\frac{\vp_{\ve}}{h}}(-h^2\Delta)^m e^{-\frac{\vp_{\ve}}{h}}u\RM_{H^s_{\mathrm{scl}}}, \mbox{ for all } u\in C^{\infty}_0(\Omega).
	\end{equation}

Next to get the required Carleman estimate for $\mathcal{L}(x,D)$, we add the lower order perturbations (given in \eqref{generaloperator}) to the above estimate. Similar literature on convexified Carleman estimates can be found in \cite{DOS} in the context of a magnetic Schr\"{o}dinger operator.

We first add the zeroth order term $(h^{2m}q)$, where $q\in L^{\infty}(\Omega,\mathbb{C})$ in the above estimate. We get 
\[
\lVert h^{2m} qu\rVert _{H^{s}_{\mathrm{scl}}} \leq h^{2m}\lVert q\rVert_{L^{\infty}} \lVert u\rVert_{L^2} \leq h^{2m} \lVert q\rVert_{L^{\infty}}\lVert u\rVert_{H^{s+2m}_{\mathrm{scl}}}.
\]
Next we consider the first order term $h^{2m}A_{\A}D^{\A}$, with $\lvert \A \rvert=1$ (here $A_{\A}\in W^{2,\infty}(\Omega,\Cb^n)$). 
We observe
\begin{equation}\label{First order term}
h^{2m-1}e^{\frac{\vp_{\ve}}{h}}\sum\limits_{\lvert\A\rvert=1}A_{\A}(hD)^{\A} e^{-\frac{\vp_{\ve}}{h}}u
= h^{2m-1}\sum\limits_{\lvert\A\rvert=1}A_{\A}(-D^{\A}\vp_{\ve}+ hD^{\A})u.
\end{equation}
The first term in the right hand side of the above expression can be estimated as 
\begin{equation*} \lVert(A_{\A}D^{\A}\vp_{\epsilon})u\rVert_{H^{s}_{\mathrm{scl}}} \leq \lVert A_{\A} D^{\A}\vp_{\epsilon}\rVert_{L^{\infty}}\lVert u\rVert_{H^{s+2m}_{\mathrm{scl}}}
\end{equation*}
Now as $\vp_{\epsilon} = \vp + \frac{h}{\epsilon}\vp^2 $ and $0 < h \ll \epsilon \ll 1$, that is, $0 < \frac{h}{\epsilon} < 1 $. Hence $ \lVert D^{\A} \vp_{\epsilon}\rVert_{L^{\infty}} = \Oc(1)$ for any $\A$ and consequently we get
\[ \lVert(A_{\A}D^{\A}\vp_{\epsilon})u\rVert_{H^{s}_{\mathrm{scl}}} \leq \mathcal{O}(1)\lVert u\rVert_{H^{s+2m}_{\mathrm{scl}}}.\]
For the second term in the right hand side of \eqref{First order term} we observe that
\begin{equation*}\begin{aligned}
\lVert A_{\A} (hD)^{\A}u\rVert_{H^{s}_{\mathrm{scl}}} &\leq \lVert hD^{\A}(A_{\A}u)\rVert_{H^{s}_{\mathrm{scl}}} + h \lVert (D^{\A}A_{\A})u\rVert_{H^{s}_{\mathrm{scl}}}\\
&\leq \Oc(1)\lVert A_{\A}u\rVert_{H^{s+1}_{\mathrm{scl}}} + \Oc(h)\lVert u\rVert_{H^{s+2m}_{\mathrm{scl}}}\\ 
&\leq \Oc(1)\lVert u\rVert_{H^{s+2m}_{\mathrm{scl}}}.
\end{aligned}\end{equation*}
The last inequality follows from the fact that the operator given as multiplication by $A_{\alpha}$ is continuous from $H^{s+2m}_{\mathrm{scl}}$ to $H^{s+1}_{\mathrm{scl}}$ where $A_\alpha\in W^{1+|\alpha|,\infty}$, $|\A|\geq 1$. 
Therefore for $\lvert \A \rvert = 1$ we have 
\begin{equation}\label{newline1}
\lVert h^{2m-1}e^{\frac{\vp_{\epsilon}}{h}}(A_{\A} hD^{\A}) e^{-\frac{\vp_{\epsilon}}{h}}u\rVert_{H^s_{\mathrm{scl}}} \leq \Oc(h^{2m-1})\lVert u\rVert_{H^{s+2m}_{\mathrm{scl}}}.    
\end{equation}
	
Now consider the term $h^{2m}A_{\A}D^{\A}e^{\varphi_{\epsilon}/h}u$, for $\lvert \A \rvert = 2$.
\begin{align}\label{2nd order term}
&h^{2m-2}e^{\frac{\vp_{\ve}}{h}}\sum\limits_{\lvert\A\rvert=2}A_{\A}h^2D^{\A} e^{-\frac{\vp_{\ve}}{h}}u\\
&=h^{2m-2}e^{\frac{\vp_{\ve}}{h}}\sum\limits_{\substack{\A=\A_1+\A_2,\\\lvert\A_1\rvert=\lvert\A_2\rvert=1}}A_{\A}h^2D^{\A} e^{-\frac{\vp_{\ve}}{h}}u\notag\\
&= h^{2m-2} \sum\limits_{{\lvert\A_1\rvert= \lvert\A_2\rvert=1}}A_{\A}(D^{\A_1}\vp_{\ve}D^{\A_2}\vp_{\ve}-hD^{\A_1+\A_2}\vp_{\ve} + hD^{\A_1}\vp_{\ve}D^{\A_2}+D^{\A_2}\vp_{\ve}hD^{\A_1} + h^2 D^{\A_1+\A_2})u.
\end{align}
For the first two terms in the right hand side of \eqref{2nd order term}, we get
\begin{align*}
\lVert A_\A \lb D^{\A_1}\vp_{\ve}D^{\A_2}\vp_{\ve} -hD^{\A_1+\A_2}\vp_{\ve}\rb u \rVert_{H^s_{\mathrm{scl}}}
&\leq C\lVert A_\A \lb D^{\A_1}\vp_{\ve}D^{\A_2}\vp_{\ve} -hD^{\A_1+\A_2}\vp_{\ve}\rb \rVert_{L^{\infty}}\lVert u\rVert_{H^s_{\mathrm{scl}}}\notag\\
&\leq \mathcal{O}(1)\lVert u\rVert_{H^{s+2m}_{\mathrm{scl}}}. 
\end{align*}
Analyzing the third term in \eqref{2nd order term}, we see
\begin{align*}
\lVert A_\A D^{\A_1}\vp_{\ve}hD^{\A_2}u \rVert_{H^s_{\mathrm{scl}}}
\leq& C\lVert hD^{\A_2}\left(A_\A u D^{\A_1}\vp_{\ve}\right) \rVert_{H^{s}_{\mathrm{scl}}} 
	+ Ch\lVert D^{\A_2}\left(A_{\A}D^{\A_1}\vp_{\ve}\right) \rVert_{L^{\infty}}\lVert u\rVert_{H^{s+2m}_{\mathrm{scl}}} \\
\leq& \mathcal{O}(1)\lVert \left(A_{\A} D^{\A_1}\vp_{\ve}\right) u\rVert_{H^{s+1}_{\mathrm{scl}}}
	+ Ch\lVert u\rVert_{H^{s+2m}_{\mathrm{scl}}}\\
\leq& \mathcal{O}(1)\lVert u\rVert_{H^{s+2m}_{\textrm{scl}}}.
\end{align*}
where, for the first term we use the continuity of the  multiplication operator $A_{\A} : H^{s+2m}_{\mathrm{scl}}\to H^{s+1}_{\mathrm{scl}}$ whenever $A_\alpha\in W^{1+|\alpha|,\infty}$, $|\A|\geq 1$.

Now, consider the last term of the expression on the right hand side of \eqref{2nd order term}, we get 
\begin{align*}
\lVert A_{\A}h^2D^{\A}u \rVert_{H^s_{\mathrm{scl}}} 
&\leq\lVert h^2D^{\A}(A_\A u) \rVert_{H^s_{\mathrm{scl}}} 
+ 2\lVert h^2D^{\A_1}(A_\A) D^{\A_2}u \rVert_{H^s_{\mathrm{scl}}} 
+\lVert h^2D^{\A}(A_\A) u \rVert_{H^{s}_{\mathrm{scl}}}\\
&\leq \mathcal{O}(1)\lVert A_\A u \rVert_{H^{s+2}_{\mathrm{scl}}} + \mathcal{O}(h)\lVert u \rVert_{H^{s+2m}_{\mathrm{scl}}} + \mathcal{O}(h^2)\lVert u \rVert_{H^{s+2m}_{\mathrm{scl}}}\\
&\leq \mathcal{O}(1)\lVert  u \rVert_{H^{s+2m}_{\mathrm{scl}}} + \mathcal{O}(h)\lVert u \rVert_{H^{s+2m}_{\mathrm{scl}}} + \mathcal{O}(h^2)\lVert u \rVert_{H^{s+2m}_{\mathrm{scl}}}.
\end{align*}
Here in the first term we use the continuity of the multiplication operator $A_{\A} : H^{s+2m}_{\mathrm{scl}}\to H^{s+2}_{\mathrm{scl}}$ whenever $A_\alpha\in W^{1+|\alpha|,\infty}$, $|\A|\geq 2$. 
To prove the continuity, it suffices to consider the complex interpolation for the cases $s = 0$ and $s = -2m$. 
The inequality on the second term on the right hand side follows as before by using the continuity of the  multiplication operator $A_{\A} : H^{s+2m}_{\mathrm{scl}}\to H^{s+1}_{\mathrm{scl}}$ whenever $A_\alpha\in W^{1+|\alpha|,\infty}$, $|\A|\geq 1$.  

Then adding all the lower order terms up to order 2 in \eqref{carlemanestimate22} and choosing $h\ll \epsilon \ll 1$ small enough and using the standard bounds i.e. 
$1\leq e^{\frac{\vp^2}{2\epsilon}}\leq C, \hspace{3pt} \frac{1}{2}\leq  1 + \frac{h}{\epsilon}\vp \leq \frac{3}{2}$ 
we get our desired estimate \eqref{carlemanestimate3}.
\end{proof}

Let us denote 
\[
\Lc_{\vp}(x,D) = h^{2m}e^{\frac{\vp}{h}}\Lc(x,D)e^{-\frac{\vp}{h}}.
\]
The formal $L^{2}$ adjoint of $\Lc_{\vp}(x,D)$ would be ${\Lc}^{*}_{\vp}(x,D)=h^{2m}e^{-\frac{\vp}{h}}\Lc^{*}(x,D)e^{\frac{\vp}{h}}$, where $\Lc^{*}(x,D)$ is the formal $L^2$-adjoint of the operator $\Lc(x,D)$. 
As $\Lc^{*}(x,D)$ has the similar form as $\Lc(x,D)$ with the same regularity of the coefficients, since $-\vp$ is also a limiting Carleman weight if $\vp$ is, the Carleman estimate derived in Proposition \ref{Prop: Interior Carleman Estimate} holds for ${\Lc}^{*}_{\vp}(x,D)$ as well.
The following proposition establishes an existence result for an inhomogeneous equation analogous to the results in \cite{KRU2,KRU1}.
\begin{proposition}\label{existence}
Let $A_{\A}\in W^{1+\lvert\A\rvert,\infty}(\Omega,\mathbb{C}^{n^{|\alpha|}})$, $|\A|=1,2$ and  $q\in L^{\infty}(\O,\Cb)$ and $\vp$ be any limiting Carleman weight for the 
semiclassical Laplacian on $\wt{\O}$. 
For $0< h\ll 1$ sufficiently small, the equation 
\begin{equation*}
\Lc_{\vp}(x,D)u = v \mbox{ in }\Omega,
\end{equation*}
has a solution $u\in H^{2}_{\mathrm{scl}}(\Omega)$, for $v\in H^{2-2m}_{\mathrm{scl}}$ 
satisfying, 
\begin{equation*}
{h}^m\lVert u \rVert_{H^{2}_{\mathrm{scl}}} \leq C \lVert v \rVert_{H^{2-2m}_{\mathrm{scl}}}.
\end{equation*}
The constant $C>0$ is independent of $h$ and depends only on $A_\A$, $\A=1,2$ and $q$.
\end{proposition}
Here we skip the proof of the above proposition as it is standard in the literature of Calder\'{o}n type inverse problems (see \cite{KRU2,KRU1}).

\subsection{Construction of C.G.O. solutions}\label{ccs}

Now we construct complex geometric optics type solutions of the equation $\Lc(x,D)u=0$ based on Proposition \ref{existence}. 
We propose a solution of the form
\begin{equation}\label{cgo} 
u = e^{\frac{(\vp + \I\psi)}{h}}(a_0(x) + ha_1(x)+r(x;h)),
\end{equation}
where $0<h\ll 1$, $\vp(x)$ is a limiting Carleman weight for the semiclassical Laplacian. 
The real valued phase function $\psi$ is chosen such that $\psi$ is smooth near $\overline{\Omega}$ and solves the following Eikonal equation $p_{0,\vp}(x,\nabla\psi) =0$ in $\wt{\O}$. 
The functions
$a_0, a_1$ are the complex amplitudes which solve certain transport equations,  which we will define later. 
The function $r$ is the correction term which satisfies the following estimate $\lVert r\rVert_{H^{2}_{\mathrm{scl}}} = \Oc(h^2)$.

We consider $\varphi$ and $\psi$ to be 
\begin{equation}\label{phi and psi}
\vp(x)=\o \cdot x,\quad 
\psi(x)= \wt{\o} \cdot x,
\end{equation}
where $\o,\wt{\o} \in \Rb^n$ are such that $\o\cdot\wt{\o} = 0$ and $\lvert \wt{\o} \rvert = \lvert \o \rvert$. Observe that $\varphi$ and $\psi$ solves the Eikonal equation $p_{0,\vp}(x,\nabla\psi) =0$ in $\wt{\O}$, that is $\Lm\nabla \vp\Rm = \Lm\nabla \psi\Rm$ and $\nabla\vp\cdot\nabla\psi =0$.

\begin{proposition}\label{solvibility}
Let us consider the equation
\begin{equation}\label{L-sharp}
\Lc(x,D)u =  (-\Delta)^m u + \sum_{\lvert\A\rvert = 1,2} {A}_{\A} D^{\A}u + qu =0,
\end{equation}
where $A_\alpha\in W^{1+\lvert\A\rvert,\infty}(\Omega,\mathbb{C}^{n^{|\alpha|}})$, $|\A|=1,2$ and  $q\in L^{\infty}(\Omega,\mathbb{C})$. We assume that for $m=2$, $A_\alpha$ with $|\alpha|=2$ is an isotropic matrix.
Then for all $0< h \ll 1$, there exists a solution $u\in H^{2}(\O)$ of \eqref{L-sharp} of the form
\begin{equation}\label{qw10} 
u(x,h) = e^{\frac{(\vp(x) + i\psi(x))}{h}}(a_0(x)+ha_1(x) + r(x;h)) 
\end{equation}
where $\vp$ and $\psi$ are as in \eqref{phi and psi}. Here $a_0$, $a_1$ are complex amplitudes satisfying certain transport equations and $r\in H^2(\Omega)$ satisfies the estimate $\lVert r\rVert_{H^{2}_{\mathrm{scl}}} = \Oc(h^2)$.
\end{proposition}

\bpr
Let us write $T= \left[(\nabla\vp +\I  \nabla\psi)\cdot \nabla\right]$ and consider the conjugated operator
\begin{align}\label{ConjugatedOperator}
&e^{-\frac{(\vp+ \I\psi)}{h}}h^{2m}\Lc(x,D)e^{\frac{(\vp + \I\psi)}{h}}u\notag\\
&\qquad\qquad= (-h^2\Delta -2hT)^mu\notag\\
&\qquad\qquad\quad+ \sum_{\substack{\A=\A_1+\A_2,\\\lvert\A_1\rvert =\lvert\A_2\rvert =1}}h^{2m-2} A_{\A}\left( D^{\A_1}(\varphi+i\psi)D^{\A_2}(\varphi+i\psi) + 2h D^{\A_1}(\varphi+i\psi)D^{\A_2} + h^2D^{\A} \right)u\notag \\
&\qquad\qquad\quad + \sum_{\lvert\A\rvert = 1} h^{2m-1} A_{\A}\left( D^{\A}(\varphi + i\psi) + hD^{\A}\right)u + h^{2m}qu.
\end{align}

\noindent Let us substitute the form of $u$ in \eqref{cgo} in the Equation \ref{ConjugatedOperator}.
We would like to get $e^{-\frac{(\vp+ i\psi)}{h}}h^{2m}\Lc(x,D)\left(e^{\frac{(\vp + i\psi)}{h}}r(x,h)\right)$ is of $\mathcal{O}(h^{m+2})$. 
In order to obtain that it is enough to consider the coefficient of $h^m$ and $h^{m+1}$ in terms of the complex amplitudes $a_0$, $a_1$ in the above expression \eqref{ConjugatedOperator} and equate them to zero, which will give two transport equations.
Let us begin with considering the coefficient of $h^m$. We get the following transport equation for $a_0$ as 
\begin{equation}\label{transportequation1}
T^{m}a_0 = 0 \quad\mbox{in }\Omega, \quad\mbox{whenever }m>2,
\end{equation} 
and
\begin{equation}\label{m=2}
\left((-2T)^m  +  \sum_{\substack{\A=\A_1+\A_2,\\\lvert\A_1\rvert =\lvert\A_2\rvert =1}}A_\alpha D^{\A_1}(\varphi+i\psi)\cdot D^{\A_2}(\varphi+i\psi)\right) a_0 =0 \quad\mbox{in }\Omega, \quad\mbox{whenever }m=2.
\end{equation}

As we have assumed that $A_\alpha$ with $|\A|=2$ is an isotropic matrix for $m=2$, so the term  $\sum_{\A=\A_1+\A_2,\lvert\A_1\rvert =\lvert\A_2\rvert =1}A_\alpha D^{\A_1}(\varphi+i\psi)\cdot D^{\A_2}(\varphi+i\psi)$ becomes zero in \eqref{m=2}. It follows from the fact that if $A_\alpha=mI$ then  \[\sum_{\substack{\A=\A_1+\A_2,\\ \lvert\A_1\rvert =\lvert\A_2\rvert =1}}A_\alpha D^{\A_1}(\varphi+i\psi)\cdot D^{\A_2}(\varphi+i\psi)= m ({\o}+{i}\wt{\o})\cdot({\o}+{i}\wt{\o})=0.\]
 
Hence, we get \eqref{transportequation1} as the transport equation of $a_0$ for all $m\geq 2$. The solution of this homogeneous transport equation $T^{m}a_0 = 0$ in $\Omega$ for $m\geq 2$ always exists and can be taken in $C^\infty(\overline{\O})$ class, see \cite{KRU2,KRU1}. 

Next we consider the coefficient of $h^{m+1}$ and obtain the following transport equation of $a_1$ as
\begin{align}\label{qw11}
(-2T)^ma_1 =& -\sum_{k=0}^{m-1}\lb (-2T)^{k}\circ (-\Delta)\circ (-2T)^{m-1-k}\rb a_0 \notag\\
&+ \sum_{\substack{\A=\A_1+\A_2\\ \lvert\A_1\rvert =\lvert\A_2\rvert =1}} A_{\A}\, D^{\A_1}(\varphi+i\psi)D^{\A_2}(\varphi+i\psi)\, a_0 \quad\mbox{in }\Omega,
\quad\mbox{whenever }  m>2
\end{align}
and for $m=2$,
\begin{align}\label{qw12}
(-2T)^ma_1 =& -\sum_{k=0}^{m-1}\lb (-2T)^{k}\circ (-\Delta)\circ (-2T)^{m-1-k}\rb a_0\notag\\
& + 2\sum_{\substack{\A=\A_1+\A_2\\ \lvert\A_1\rvert =\lvert\A_2\rvert =1}}A_\A D^{\A_1}(\varphi+i\psi)D^{\A_2}a_0 + \sum_{\lvert\A\rvert = 1} A_{\A}D^{\A}(\varphi + i\psi)\, a_0 \quad\mbox{in }\Omega.
\end{align}
The above inhomogeneous transport equations are solvable and we solve them in $W^{2,\infty}(\Omega)$ to ensure $u$ in \eqref{qw10} is indeed in $H^{2}(\Omega)$.   
In order to show the solvability of \eqref{qw11}, \eqref{qw12} in $W^{2,\infty}(\Omega)$ one can break it into a system of $m\geq 2$ linear equations as follows.
\begin{align*}
&\mbox{Given }
f\in W^{2,\infty}(\Omega),
\mbox{ find }
v_1\in W^{2,\infty}({\Omega}) \mbox{ solving }Tv_1 =f \mbox{ in }\Omega;\\
&\mbox{Given }
v_1\in W^{2,\infty}({\Omega})
,
\mbox{ find }
v_2\in W^{2,\infty}({\Omega})
\mbox{ solving }
Tv_2 =v_1 \mbox{ in }\Omega;\\
&\mbox{Proceeding as before, given }
v_{m-1}\in W^{2,\infty}({\Omega}) \mbox{ find } a_1\in W^{2,\infty}({\Omega})
\mbox{ solving } Ta_1 =v_{m-1} \mbox{ in }\Omega.
\end{align*}

The solvability of the linear inhomogeneous equation
\[Tg=\left(\omega+i\wt{\omega}\right)\cdot\nabla g=f\quad\mbox{ in }\Omega\] for $\omega\perp\wt{\omega}=0$ and $\lvert \omega\rvert =\lvert\wt{\omega}\rvert=1$ is well known, see \cite{SALOT}. 
Let $f\in W^{k,\infty}({\Omega},\mathbb{C})$, $k\geq 0$ then we get a solution $g\in W^{k,\infty}({\Omega},\mathbb{C})$ as
\[ g(x) = \int_{\mathbb{R}^2} \frac{\widetilde{f}(x-\omega y_1-\wt{\omega}y_2)}{y_1+iy_2}dy_1dy_2\]
where $\widetilde{f}\in W^{k,\infty}(\mathbb{R}^n;\mathbb{C})$ with $\mathrm{supp}\,(\wt{f})$ is a compact set containing $\Omega$ and $\wt{f}=f$ in $\Omega$.

We considered $a_0$ in $C^\infty(\overline{\Omega})$ and for $a_1$ to be in at least in $W^{2,\infty}({\Omega})$ we require $A_\A \in W^{2,\infty}({\Omega})$ for $|\A|=1,2$ whenever $m\geq2$. 
As we have assumed $A_\A\in W^{1+|\A|,\infty}(\Omega)$ for $\A=1,2$ hence it justifies $a_1\in W^{2,\infty}(\O)$ and hence it is in $H^2(\Omega)$.

The remainder term $r(x,h)$ satisfies
\begin{equation}\label{correctorequation}\begin{aligned}
e^{-\frac{(\vp+ i\psi)}{h}}h^{2m}\Lc(x,D)\left(e^{\frac{(\vp + i\psi)}{h}}r(x,h)\right) &= -e^{-\frac{(\vp+ i\psi)}{h}}h^{2m}\Lc(x,D)\left(e^{\frac{(\vp + i\psi)}{h}}(a_0+ha_1)\right) \\
&=\Oc(h^{m+2}).
\end{aligned}\end{equation}
By choosing $a_0\in C^{\infty}(\overline{\Omega})$ and $a_1\in W^{2,\infty}(\Omega)$ the reminder term $r$ solves
\begin{equation*}
e^{-\frac{\vp}{h}}h^{2m}\Lc(x,D)e^{\frac{\vp}{h}}(e^{\frac{i\psi}{h}}r(x;h)) = \Oc(h^{m+2}) \quad \mbox{in }H^{2-2m}_{scl}(\Omega),
\end{equation*}
which is solvable thanks to the Proposition \ref{existence}.
We get a solution $r\in H^{2}(\Omega)$ satisfying the estimate $\lVert r\rVert_{H^{2}_{\mathrm{scl}}} = \Oc(h^2)$. 
\epr

\section{Determination of the coefficients}\label{determination}
Our first step is the standard method of extending the problem to a larger simply connected domain. Similar to \cite{KRU1,SU} we have the following result. 
\begin{proposition}\label{qw20}
Let $\Omega \ssubset \wt{\O}$ be two bounded domains in $\mathbb{R}^n$ with smooth
boundaries, and let $A_\A,\, \wt{A}_\A \in W^{1+|\A|,\infty}(\wt{\O},\mathbb{C}^{n^{|\A|}})$ for $|\A|=1,2$, and  $q,\, \wt{q} \in L^\infty(\O,\mathbb{C})$ satisfy 
$A_\A = \wt{A}_\A$ for $|\A|=1,2$ and $q = \wt{q}$ in $\wt{\O}\setminus\O$. If the Cauchy data set (cf. \eqref{qw19}) $\mathcal{C}_{A_\A,q}(\O) = \mathcal{C}_{\wt{A}_\A,\wt{q}}(\O)$, then $\mathcal{C}_{A_\A,q}(\wt{\O}) = \mathcal{C}_{\wt{A}_\A,\wt{q}}(\wt{\O})$.
\end{proposition}
The proof of the above proposition follows from a standard analysis technique, which is given in \cite{KRU1,SU}.

We recall that 
$A_\A,\wt{A}_\A 
\in W^{1+\lvert\A\rvert,\infty}(\mathbb{R}^n,\mathbb{C}^{n^{|\alpha|}}) \cap \mathcal{E}^{\prime}(\overline{\Omega})$ for $|\A|=1,2$ and $q,\wt{q}\in L^\infty(\O)$.
Let $\wt{\Omega}$ be a smooth simply connected domain in $\mathbb{R}^n$ such that  $\O\ssubset\wt{\O}$. We extend all the lower order perturbed coefficients $A_\A,\wt{A}_\A,q,\wt{q}$ by zero to $\wt{\O}\setminus\O$, and denote these extensions by the same letters. Now using the Proposition \ref{qw20} we get  $\mathcal{C}_{A_\A,\, q}(\wt{\O}) = \mathcal{C}_{\wt{A}_\A,\,\wt{q}}(\wt{\O})$.
Therefore now we can study the problem in a bigger domain $\wt{\Omega}$ which is simply connected. 

\subsection{Integral identity involving the coefficients $A_{\A}$, $q$}
We recall that
\[\Lc(x,D) \equiv 
 (-\Delta)^m + \sum\limits_{\lvert \A \rvert = 1}^{2} A_{\A}(x)D^{\A} + q(x),
\]
where $A_{\A}\in W^{1+\lvert \A \rvert,\infty}(\wt{\Omega},\mathbb{C}^{n^{|\alpha|}})$ for $\lvert\A\rvert = 1,2$ and $q \in L^{\infty}(\wt{\Omega},\mathbb{C})$.
We write the formal $L^{2}$ adjoint of this operator, $\Lc^{*}(x,D)$ is of the form
\begin{equation}\begin{aligned}\label{adjoint-operator}
\Lc^{*}(x,D)\equiv 
 (-\Delta)^m + \sum\limits_{\lvert \A \rvert = 1}^{2}
 A^{\sharp}_{\A}(x)D^{\A} + q^{\sharp}(x),
\end{aligned}\end{equation}
where 
\[\begin{cases}
A^{\sharp}_{\A}(x) &= A_{\A}(x), \quad \mbox{for } \lvert \A \rvert = 2,\\ 
A^{\sharp}_{\A}(x) &= \sum_{\lvert\B\rvert=2} D^{\B-\A}\overline{A}_{\B}(x),  \quad \mbox{for } \lvert \A \rvert = 1\\
q^{\sharp}(x) &= \sum_{\lvert\B\rvert = 1}^{2} D^{\A}\overline{A}_{\B}(x) + \overline{q}(x).
\end{cases}\]

\noindent
We have the following integral identity 
\begin{equation}\label{integralidentity}
\int_{\wt{\Omega}} \left(\Lc(x,D)u\right)\overline{v} \D x - \int_{\wt{\Omega}} u\overline{\Lc^{*}(x,D)v}\D x = 0, \quad \forall u\in H^{2m}_0(\wt{\O}), v\in H^{2m}(\wt{\O}).
\end{equation}
where $H^{2m}_0(\wt{\Omega})$ is the closure of $C_0^\infty(\wt{\Omega})$ functions in $H^{2m}(\wt{\Omega})$ norm.

\noindent
Let $m\geq 2$ and $u,\wt{u}\in H^{2m}(\wt{\Omega})$ solves
\begin{equation}\begin{aligned}
\Lc(x,D)u &=0\quad\mbox{in }\wt{\Omega}\qquad \quad
\mbox{and}\quad
\widetilde{\Lc}(x,D)\widetilde{u} =0 \quad\mbox{in }\wt{\Omega},\\
\mbox{with }\quad (-\Delta)^l u|_{\partial\wt{\Omega}} &= (-\Delta)^l \widetilde{u}|_{\partial\wt{\Omega}},\quad \mbox{for }l=0,1,\ldots,(m-1).
\end{aligned}
\end{equation}
From the assumption of the Theorem \ref{mainresult} and Proposition \ref{qw20} on $\partial\wt{\O}$ we now get 
\begin{equation}
\partial_\nu(-\Delta )^{l}u = \partial_\nu(-\Delta )^{l}\widetilde{u},\qquad\mbox{for }l=0,2,\ldots,(m-1).
\end{equation}
So we have $(u-\widetilde{u})\in H^{2m}_0(\wt{\Omega})$.
Therefore
\begin{equation}\label{differenceEven}
\Lc(x,D)(u-\widetilde{u})= \sum_{1\leq\lvert \A \rvert \leq2}(A_{\A}-{\wt{A}_\A})D^{\A}\widetilde{u} +  (q-\widetilde{q})\widetilde{u}.
\end{equation}
Let $v\in H^{2m}(\wt{\Omega})$ satisfies $\Lc^{*}(x,D)v = 0 $ in $\wt{\Omega}$,
then from the integral identity \eqref{integralidentity} we get
\begin{equation}\begin{aligned}\label{integralidentityE}
\int_{\wt{\Omega}}\lb \sum_{1\leq\lvert \A \rvert \leq 2}(A_{\A}-{\wt{A}_\A})D^{\A}\widetilde{u} + (q-\widetilde{q})\widetilde{u}\rb\overline{v}\,  \D x
= 0
\end{aligned}\end{equation}
Next we choose $\wt{u}$ and $v$ to be the CGO type solutions constructed in the previous subsection \ref{ccs}. 
We choose $\varphi=\omega\cdot x$ and $\psi=\widetilde{\omega}\cdot x$ for $\wt{u}$ and $\varphi=-\omega\cdot x$ and $\psi=\wt{\omega}\cdot x$ for $v$, where  
$\o,\wt{\o} \in \mathbb{R}^n$ satisfying $|\o|=|\wt{\o}|=1$ and $\o\cdot\wt{\o}=0$. For $h>0$ small enough, we set the solutions are of the form
\begin{equation}
\label{cgo2}
\begin{cases}
\widetilde{u}(x)=&e^{\frac{\omega\cdot x+ i\wt{\omega}\cdot x}{h}} \lb\wt{a_0}(x,\o+i\wt{\o})+ h\wt{a_1}(x,\o+i\wt{\o})+\widetilde{r}(x,\omega+ i\wt{\omega};h)\rb\quad\mbox{in }\wt{\O},\\
v(x)=&e^{\frac{-\omega\cdot x+i\wt{\omega}\cdot x}{h}} \lb a_0(x,-\o+i\wt{\o})+ha_1(x,-\o+i\wt{\o})+ r(x,-\omega+i\wt{\omega};h)\rb\quad\mbox{in }\wt{\O}.
\end{cases}
\end{equation}
The amplitudes $\wt{a_0}(\cdot,\o+i\wt{\o})$, ${a}_0(\cdot,-\o+i\wt{\o})\in C^\infty(\overline{\wt{\O}})$ satisfy the transport equations 
\begin{equation}
\label{eq_transp_new_1}
\begin{cases}
\lb(\omega+i\wt{\omega})\cdot \nabla\rb^m\, \wt{a_0}(x,\omega+i\wt{\omega})=0\quad \textrm{in}\ \wt{\O},\\   
\lb(-\omega+i\wt{\omega})\cdot \nabla\rb^m\, a_0(x,-\omega+i\wt{\omega})=0\quad \textrm{in}\ \wt{\O},
\end{cases}
 \quad m\geq 2
\end{equation}
and 
\begin{equation}
\label{eq_remainder_r_j}
\|\wt{r}\|_{H^{2}_{\textrm{scl}}},\, \|r\|_{H^{2}_{\textrm{scl}}}=\mathcal{O}(h^2).
\end{equation}

Now substituting \eqref{cgo2} in \eqref{integralidentityE} we get
\begin{equation}\label{cgoIdentityE}
\begin{aligned}
0=&\sum_{\substack{\A=\A_1+\A_2,\\ \lvert\A_1\rvert =\lvert\A_2\rvert= 1}}\int_{\wt{\Omega}} (A_{\A}-{\wt{A}_\A})\frac{(\o + i\wt{\o})_{\A_1}}{h}\frac{(\o + i\wt{\o})_{\A_2}}{h}
\lb \wt{a_0} +h\wt{a_1} + \wt{r}\rb\overline{\lb {a}_0 + ha_1 + {r}\rb }\ \D x\\ 
&+\sum_{\substack{\A=\A_1+\A_2,\\ \lvert\A_1\rvert =\lvert\A_2\rvert= 1}}\int_{\wt{\Omega}} (A_{\A}-{\wt{A}_\A})\frac{(\o + i\wt{\o})_{\A_1}}{h}
D^{\A_2}\lb \wt{a_0} + h\wt{a_1} + \wt{r}\rb\overline{\lb {a}_0 +ha_1+ {r}\rb }\ \D x\\
&+\sum_{\lvert\A\rvert =2}\int_{\wt{\Omega}} (A_{\A}-{\wt{A}_\A})
D^{\A}\lb \wt{a_0}+h\wt{a_1} + \wt{r}\rb\overline{\lb {a}_0 +ha_1 + {r}\rb }\ \D x\\
&+\sum_{\Lm \A \Rm = 1} \int_{\wt{\Omega}} (A_{\A}-{\wt{A}_\A})\frac{(\o + i\wt{\o})_{\A}}{h}
\lb \wt{a_0}+h\wt{a_1} + \wt{r}\rb\overline{\lb {a}_0 + ha_1 + {r}\rb }\ \D x\\
&+\sum_{\Lm \A \Rm = 1} \int_{\wt{\Omega}} (A_{\A}-{\wt{A}_\A})D^\alpha
\lb \wt{a_0} +h\wt{a_1} + \wt{r}\rb\overline{\lb {a}_0 +ha_1 + {r}\rb }\ \D x \\
&+\int_{\wt{\Omega}} (q-\widetilde{q})\lb \wt{a_0} +h\wt{a_1} +
\wt{r}\rb \overline{\lb {a}_0 + ha_1+ {r}\rb}\ \D x.
\end{aligned}
\end{equation}

Now we consider following two cases for the different forms of the second order perturbation.\\
\textbf{Case(1):}
Let us assume that $m>2$ and $(A_\A-\wt{A}_\A)$ for $|\A|=2$ is not an isotropic matrix. 
Then multiplying \eqref{cgoIdentityE} by $h^{2}$ and letting $h\to 0$ we get
\begin{equation}\label{limitingidentityE}
\begin{aligned}
\sum_{\substack{\A=\A_1+\A_2,\\ \lvert\A_1\rvert =\lvert\A_2\rvert= 1}}\int_{\wt{\Omega}} (A_{\A}-{\wt{A}_\A})(\o + i\wt{\o})_{\A_1}(\o + i\wt{\o})_{\A_2}
\,\wt{a_0}\,\overline{{a}_0}\ \D x=0.
\end{aligned}
\end{equation}
This follows from the fact that $A_\A, {\wt{A}_\A}\in L^\infty(\wt{\Omega})$, $\wt{a_0},\, a_0\in C^\infty(\overline{\wt{\O}})$ and  $\wt{a_1},\, a_1\in W^{2,\infty}(\wt{\Omega})$. 
Note that here we have crucially used the fact that $\|\wt{r}\|_{H^{2}_{\textrm{scl}}},\, \|r\|_{H^{2}_{\textrm{scl}}}
=\mathcal{O}(h^2)$ to obtain $\|\wt{r}\|_{L^2},\, \|r\|_{L^2} =\mathcal{O}(h^2)$, $\|D^{\beta}\wt{r}\|_{L^2}=\mathcal{O}(h)$, for $|\beta|=1$ and $\| D^{\A}\wt{r}\|_{L^2}=\mathcal{O}(1)$, for $|\A|=2$.
\vspace{5mm}

\noindent\textbf{Case(2):} 
Let us assume that $(A_\A-\wt{A}_\A)$ for $|\A|=2$ is isotropic. 
Observe that the first term in \eqref{cgoIdentityE} is always zero. 
Hence, multiplying \eqref{cgoIdentityE} by $h$ and letting $h\to 0$, we get 
\begin{equation}\label{iso}
\begin{aligned}
\sum_{\substack{\A=\A_1+\A_2,\\ \lvert\A_1\rvert =\lvert\A_2\rvert= 1}}\int_{\wt{\Omega}} (A_{\A}-{\wt{A}_\A})(\o + i\wt{\o})_{\A_1} D^{\A_2}\wt{a_0}\ \overline{{a}_0}\ \D x 
+\sum_{\lvert\A\rvert = 1}\int_{\wt{\Omega}} (A_{\A}-{\wt{A}_\A})(\o + i\wt{\o})_{\A} \wt{a_0}\,\overline{{a}_0}\ \D x =0,
\end{aligned}
\end{equation}
It follows from the fact that $A_\A, {\wt{A}_\A}\in L^\infty(\wt{\Omega})$, $\wt{a_0},\, a_0\in C^\infty(\overline{\O})$ and  $\wt{a_1},\, a_1\in W^{2,\infty}(\wt{\Omega})$.
Note that here we use $\|\wt{r}\|_{L^2},\, \|r\|_{L^2} =\mathcal{O}(1)$, $\|hD^{\beta}\wt{r}\|_{L^2}=\mathcal{O}(h)$ with $|\beta|=1$, and $\|h^j D^{\A}\wt{r}\|_{L^2}=\mathcal{O}(h^j)$ with $|\A|=2$, $j=1,2$, specially in the third term of \eqref{cgoIdentityE}.

\subsection{Determination of the coefficients} In this section we will establish the equality of $A_\A =\wt{A}_\A$ for $|\A|=1,2$ and $q=\wt{q}$ in $\O$ from the integral identities \eqref{limitingidentityE} and \eqref{iso}. 
\bpr[Proof of Theorem \ref{mainresult}]
Let us begin with the integral identity \eqref{limitingidentityE} in the \textbf{Case (1)} mentioned above.

In \cite{Sharafutdinov}, it is proved that a symmetric $2$-form $(A_\A-\wt{A}_\A)\in W^{3,\infty}(\wt{\Omega},\mathbb{C}^{n^2})$ can be uniquely decomposed in an $1$-form $V\in W^{4,\infty}(\wt{\Omega},\mathbb{C}^n)$ with $V = 0$ on $\partial\wt{\Omega}$ and a symmetric $2$-form $F\in W^{3,\infty}(\wt{\Omega},\mathbb{C}^{n^2})$ so that 
\begin{equation}\label{qw4}
(A_\A-\wt{A}_\A) = F + d_sV, \quad  \mbox{in } \wt{\Omega}.
\end{equation}
Here $F$ is divergence free, i.e.
\[(\delta F)_j = \sum_{k=1}^n\frac{\partial F_{jk}}{\partial x_k} = 0, \quad\mbox{in }\wt{\O},\quad j=1,2,..,n\]   and 
$d_s$ is the symmetrized differentiation defined as 
\[(d_s V)_{jk}=\frac{1}{2}\lb \frac{\partial V_j}{\partial x_k} +  \frac{\partial V_k}{\partial x_j} \rb, \quad\mbox{in }\wt{\O}.\]  
The above expression \eqref{qw4} can be realized in the following manner. Let us take any vector field $B\in W^{k,\infty}(\wt{\Omega})$, $k\geq 1$, and consider the following Laplace equation $-\Delta \varphi = div\, B $ in $\wt{\Omega}$. We can get a unique solution $\varphi \in W^{k+1,\infty}(\wt{\Omega})$ subject to suitable Dirichlet  or Neumann  boundary condition. 
For example, we get a unique $\varphi\in W^{3,\infty}(\wt{\Omega})$ by solving
\[ (-\Delta)\varphi = div\, B \mbox{ in }\wt{\Omega}, \quad\mbox{with } \varphi = 0 \mbox{ on }\partial\wt{\Omega}.\]
Then we can write $B= \wt{B} + \nabla\varphi$ in $\wt{\Omega}$, where $\wt{B} =(B-\nabla\varphi)$ with $div\, \wt{B} =0$ in $\wt{\Omega}$, and $\varphi =0 $ on $\partial\wt{\Omega}$.
Applying the above technique for each row (or column) vector of the matrix $(A_\alpha-\wt{A}_\alpha)$ and using symmetrization one can obtain the decomposition in \eqref{qw4}.

Using the decomposition \eqref{qw4} in \eqref{limitingidentityE} we get 
\begin{equation}\label{integral_id_2}
\sum_{\substack{\A=\A_1+\A_2,\\ \lvert\A_1\rvert =\lvert\A_2\rvert= 1}}\int_{\wt{\Omega}} (F + d_sV)_{\A}(\o + i\wt{\o})_{\A_1}(\o + i\wt{\o})_{\A_2}\, \wt{a_0}\overline{a_0}\, \D x = 0.
\end{equation}
We choose ${a}_0=1,\wt{a_0}=e^{-ix\cdot\xi}$, where $\xi \in \Rb^n$ is perpendicular to $\o$ and $\wt{\o}$. Then $a_0,\wt{a_0}$ are indeed solutions of the transport equations \eqref{eq_transp_new_1} 
and we get 
\begin{equation*}
\sum_{\substack{\A=\A_1+\A_2,\\ \lvert\A_1\rvert =\lvert\A_2\rvert= 1}}\int_{\wt{\Omega}} (F + d_sV)_{\A}(\o + i\wt{\o})_{\A_1}(\o + i\wt{\o})_{\A_2}\, e^{-ix\cdot\xi}\, \D x = 0.
\end{equation*}
Since, $V|_{\partial \wt{\Omega}} = 0$ and $\xi \cdot \omega = \xi \cdot \wt{\omega} = 0$ we get 
\begin{equation*}
\begin{aligned}
\int_{\wt{\Omega}} (d_sV)_{\A}(\o + i\wt{\o})_{\A_1}(\o + i\wt{\o})_{\A_2}\, e^{-ix\cdot\xi}\, \D x
= \left((\o + i\wt{\o})\cdot \xi\right)\int_{\wt{\Omega}} \left(V \cdot (\o + i\wt{\o})\right)\, e^{-ix\cdot\xi}\, \D x = 0.
\end{aligned}
\end{equation*}
Therefore,
\begin{align}\label{n1}
\sum_{\substack{\A=\A_1+\A_2,\\ \lvert\A_1\rvert =\lvert\A_2\rvert= 1}}\int_{\wt{\Omega}} F_{\A}(\o + i\wt{\o})_{\A_1}(\o + i\wt{\o})_{\A_2}\, e^{-ix\cdot\xi}\, \D x = 0.
\end{align}
Now extending $F$ by $0$ outside $\wt{\Omega}$ and denoting the extended function by the same notation $F \in L^2(\mathbb{R}^n)$ we get
\begin{align} \label{equation B}
\sum_{j,k = 1}^{n} \wh{F}_{jk}(\xi)(\o + i\wt{\o})_{j}(\o + i\wt{\o})_{k} = 0.
\end{align}
Let us now fix $\xi \in \Rb^n \setminus\{ 0 \}$. 
Consider an orthonormal basis $\{ \mu_1, \mu_2, \dots, \mu_n \}$ of $\Rb^n$ so that $\mu_n=\frac{\xi}{\lvert \xi \rvert}$.
Choose $\o=\mu_p$ and $\wt{\o}=\mu_q$ for $p\neq q$ and $1 \leq p,q \leq (n-1)$. 
Observe that the construction of the solutions $\wt{u}$ and $v$ allows us to make these choices of $\o$ and $\wt{\o}$. 
In fact, without loss of generality, the relation in \eqref{equation B} is true whenever $\xi,\o,\wt{\o} \in \mathbb{R}^n$ such that $\{\xi,\o,\wt{\o}\}$ are mutually perpendicular and $\lvert \o \rvert  = \lvert \wt{\o} \rvert = 1$. 

Now, equation \eqref{equation B} implies
\begin{align}\label{step 1}
\sum_{j,k = 1}^{n} \wh{F}_{jk}(\xi)\lb(\mu_p)_{j} ((\mu_1)_{k} - (\mu_q)_{j}(\mu_2)_{k}) + 2i((\mu_p)_{j}(\mu_q)_{k})\rb = 0.
\end{align}
Choosing $-\mu_q$ in place of $\mu_q$ we get an another {\color{blue}orthonormal } basis of $\mathbb{R}^n$ and consequently we find
\begin{align}\label{step 2}
\sum_{j,k = 1}^{n} \wh{F}_{jk}(\xi)\lb((\mu_p)_{j} (\mu_p)_{k} - (\mu_q)_{j}(\mu_q)_{k}) - 2i((\mu_p)_{j}(\mu_q)_{k})\rb = 0.
\end{align}
Adding and subtracting equations \eqref{step 1} and \eqref{step 2} for $1 \leq p,q\leq n-1$, $p\neq q$ we obtain
\begin{equation}\label{conditions for F}
\langle \wh{F}(\xi)\mu_p,\mu_p\rangle=\langle \wh{F}(\xi)\mu_q,\mu_q\rangle\quad\mbox{ and }\quad
\langle \wh{F}(\xi)\mu_p,\mu_q\rangle=0=\langle \wh{F}(\xi)\mu_q,\mu_p\rangle.
\end{equation}
Since $F$ is divergence free, we also have 
\begin{equation}\label{divfree}
\sum_{k =1}^{n}\wh{F}_{jk}(\xi)\xi_{k} = 0\quad\mbox{for each }j=1,\ldots,n.
\end{equation}
Using the symmetry of $\wh{F}(\xi)$ we obtain $\langle \wh{F}(\xi){\mu_p},{\xi}\rangle = 0$ for any $p=1,\ldots,n-1$. 
This leads to have the following representation of the element $\wh{F}(\xi)\mu_p $ with respect to orthonormal basis $\{\mu_1,\ldots,\mu_{n-1},\frac{\xi}{\lvert \xi \rvert}\}$ of $\mathbb{R}^n$ as 
\begin{equation*}
\wh{F}(\xi){\mu_p} = \sum_{j=1}^{n-1} d^{(p)}_{j}(\xi){\mu_j},\quad p=1,\ldots,n-1,
\end{equation*}
for some  $d^{(p)}_j(\xi)\in\mathbb{C}$, $j=1,\ldots,n-1$.
Using the second relation in \eqref{conditions for F} we see $d^{(p)}_j(\xi) = 0$ whenever $j\neq p$. Hence,
\begin{equation*}
\wh{F}(\xi){\mu_p} = d^{(p)}_{p}(\xi){\mu_p}, \quad p=1,\ldots,n-1.
\end{equation*}
Now, the first relation in \eqref{conditions for F} implies 
\begin{align*}
d^{(p)}_{p}(\xi) = \langle \wh{F}(\xi){\mu_p},{\mu_p}\rangle = \langle \wh{F}(\xi){\mu_q},{\mu_q}\rangle = d^{(q)}_{q}(\xi), \quad \forall p,q =1,\ldots,(n-1).
\end{align*}
Hence, we have $\wh{F}(\xi){\mu_p} = d(\xi){\mu_p}$, where 
\begin{equation}\label{n2}
d(\xi) = \langle \wh{F}(\xi) \mu_p, \mu_p \rangle = \int_{\mathbb{R}^n} \langle F(x) \mu_p, \mu_p \rangle e^{-ix.\xi} \,dx, \quad \mbox{for all } 1\leq p\leq n-1.
\end{equation}
Therefore, $\wh{F}(\xi)$ has eigenvalue $d(\xi)$ of multiplicity $(n-1)$ corresponding to the eigenvectors $\mu_1, \mu_2, \dots, \mu_{n-1}$ and eigenvalue $0$ corresponding to the eigenvector $\xi$.
Let us now define the orthonormal matrix
\begin{equation*}
P^{t} = \left( \mu_1 \quad \mu_2 \quad \dots \quad \mu_{n-1} \quad \frac{\xi}{|\xi|} \right)
\end{equation*}
and observe that $\wh{F} = P^{t}DP$, where $P^{t}$ is the transpose of the matrix $P$ and $D(\xi)$ is the diagonal matrix given as $D(\xi)=diag\, (\underbrace{d(\xi),\ldots,d(\xi)}_{n-1},0)$. 

This leads to provide a formal expression of $\wh{F}(\xi)$ for any $\xi \neq 0$ as
\begin{equation}\label{n3}
\wh{F_{jk}}(\xi) 
= \sum_{l,m =1}^{n} P_{lj}(\xi)D_{lm}(\xi)P_{mk}(\xi) 
= \sum_{l=1}^{n-1} d(\xi) P_{lj}P_{lk}
= d(\xi) (\delta_{jk} - P_{nj}P_{nk} )
= d(\xi) \left(\delta_{jk} - \frac{\xi_j \xi_k}{\lvert \xi \rvert^2} \right).
\end{equation}
In other words
\[ \widehat{F}(\xi) = d(\xi)\left(I- \frac{\xi\otimes\xi}{|\xi|^2}\right),\quad\xi\in \mathbb{R}^n\setminus \{0\}\]
and $\widehat{F}(0)$ can be taken as $\int_{\widetilde{\Omega}}F(x)\, dx$.

Consequently, in the physical space $x \in \mathbb{R}^n$ we have
\begin{equation}\label{nF}
F_{jk}(x) = d_{\#}(x) \delta_{jk} + R_j R_k (d_{\#}(x)),
\end{equation} 
where $d_{\#}\in L^2(\mathbb{R}^n)$ with $\wh{d_{\#}}(\xi) = d(\xi)$ and $R_j$ is the classical Riesz transformation defined as $\wh{R_j f}(\xi) = \frac{1}{i}\frac{\xi_j}{\lvert \xi \rvert}\wh{f}(\xi)$, for $f \in L^2(\mathbb{R}^n)$. 

Since $F_{jk}\in L^2(\mathbb{R}^n)$ supported inside $\wt{\Omega}$, therefore $\mathrm{trace}[F(x)] = 0 $ in $\mathbb{R}^n\setminus\wt{\Omega}$ . 
As $\sum_{j=1}^n R_j^2= -I$, from \eqref{nF} we get $\mathrm{trace}[F(x)] = (n-1)d_{\#}(x)$, therefore we derive $d_{\#}(x)= 0 $ in $\mathbb{R}^n\setminus \wt{\Omega}$.

Having $d_{\#}\in L^2(\mathbb{R}^n)$ with its compact support, let us consider $\widetilde{d}\in H^2(\mathbb{R}^n)$ solving uniquely 
\[ -\Delta \widetilde{d} = d_{\#} \quad\mbox{in }\mathbb{R}^n.\]
Then by using a standard property of the Riesz transform given as $R_jR_k (-\Delta) = \frac{\partial^2}{\partial x_j\partial x_k}$, from \eqref{nF} we get
\begin{equation}\label{kh1}
F_{jk}(x) = d_{\#}(x)\delta_{jk} + \frac{1}{2}\left[\frac{\partial}{\partial x_j}\left(\frac{\partial \widetilde{d}}{\partial x_k}\right) +\frac{\partial}{\partial x_k}\left(\frac{\partial \widetilde{d}}{\partial x_j}\right)\right] \quad\mbox{in }\mathbb{R}^n. 
\end{equation}
The above expression shows that $F_{jk}\in L^2(\mathbb{R}^n)$ can be written as a sum of two $L^2(\mathbb{R}^n)$ functions. 

As a next step, we show $\widetilde{d}= 0$ in $\mathbb{R}^n\setminus\wt{\Omega}$.
Since we got $d_{\#}(x)= 0 $ in $\mathbb{R}^n\setminus \wt{\Omega}$. 
Then from \eqref{kh1} we get $\frac{\partial^2 }{\partial x_j\partial x_k}\wt{d}=0$ in $\mathbb{R}^n\setminus\wt{\Omega}$ for all $j,k=1,..,n$. 
 This implies $\widetilde{d}$ is some linear function in $x$ in the domain $\mathbb{R}^n\setminus\wt{\Omega}$. However, as $\wt{d}\in H^2(\mathbb{R}^n)$ i.e.  having $L^2(\mathbb{R}^n\setminus\wt{\Omega})$-integrability of $\widetilde{d}$ in the unbounded (in all direction) domain $\mathbb{R}^n\setminus\wt{\Omega}$ that linear function in $x$ in $\mathbb{R}^n\setminus\wt{\Omega}$ could only be zero function, i.e. $\widetilde{d}=0$ in $\mathbb{R}^n\setminus \wt{\Omega}$.
Therefore, $\widetilde{d}\in H^2_0(\wt{\Omega})$. 

So our decomposition in \eqref{qw4} is now modified into
\begin{equation}\label{kh2}
\lb A-\wt{A}\rb_{jk} = (F+d_sV)_{jk}=d_{\#}(x)\delta_{jk} + \frac{1}{2}\left[\frac{\partial}{\partial x_j}\left(\frac{\partial \widetilde{d}}{\partial x_k} + V_k\right)+ \frac{\partial}{\partial x_k}\left(\frac{\partial \widetilde{d}}{\partial x_j}+ V_j\right)\right]  \quad\mbox{in }\wt{\Omega}.
\end{equation}
We denote $\widetilde{V} =\left(\nabla \widetilde{d} + V\right)\in H^1(\widetilde{\Omega})$ and observe that $\widetilde{V}|_{\partial\widetilde{\Omega}}=0$, since $\wt{d}\in H^2_0(\wt{\Omega})$ and $V|_{\partial\widetilde{\Omega}}=0$.

Next we show that $\widetilde{V}=0$ in $\wt{\Omega}$. For that we consider
\eqref{integral_id_2} with the form of $(F+d_sV)$ given in \eqref{kh2} and obtain
\begin{equation}\label{qw5}
\sum_{j,k=1}^n\int_{\wt{\Omega}} \frac{1}{2}\lb\frac{\partial \widetilde{V}_j}{\partial x_k} + \frac{\partial \widetilde{V}_k}{\partial x_j}\rb (\o + i\wt{\o})_{j}(\o + i\wt{\o})_{k}
\,\wt{a_0}\,\overline{{a_0}}\ \D x=0.
\end{equation}
We consider $\wt{a_0}= e^{-ix\cdot \xi}$ and  $a_0=(\omega\cdot x)$  in $\wt{\Omega}$, with  $\omega\cdot\o =1$ and $\o\cdot\wt{\o}=\o\cdot\xi=0=\wt{\o}\cdot\xi$, which are indeed the solutions of the transport equations \eqref{eq_transp_new_1}. 
By integration by parts with using $\wt{V}=0$ on $\partial\wt{\Omega}$, from \eqref{qw5} we obtain 
\begin{equation}\label{kh5}
\sum_{l=1}^n\int_{\wt{\Omega}} \wt{V}_l(x)\,(\o + i\wt{\o})_{l}\, e^{-ix\cdot\xi}\, \D x=0.
\end{equation}
Then following \cite{SUN,KRU1}, from \eqref{kh5} one can directly conclude that
\[ \partial_j \wt{V}_k - \partial_k \wt{V}_j = 0 \quad\mbox{in }\wt{\Omega}, \ \forall j,k=1,\ldots,n.\]
Since $\wt{\Omega}$ is simply connected, thus 
$\wt{V}=\nabla p$ for some scalar function $p\in H^{2}(\wt{\Omega})$.

Since $\wt{V}=0$ on $\partial\wt{\Omega}$, we infer $\nabla_{tan}\, p|_{\partial\wt{\Omega}}=\left(\nabla p -(\partial_{\nu}p)\nu\right)|_{\partial\Omega}=0$ and consequently get that $p$ is constant ($p|_{\partial\wt{\Omega}}=c\in\mathbb{C}$) on the boundary.
Here $\nabla_{tan}$ denotes the tangential component of $\nabla$ along the boundary $\partial \wt{\O}$.
Hence, by considering $(p-c)$ in place of $p$, we assume that $p=0$ on $\partial\wt{\Omega}$. 
Therefore, from \eqref{qw5} we get $p=\partial_\nu p=0$ on $\partial \wt{\Omega}$ and consequently we see
\begin{equation}\label{qw6}
0=\sum_{j,k=1}^n\int_{\wt{\Omega}} \frac{\partial^2 p}{\partial x_j\partial x_k} (\o + i\wt{\o})_{j}(\o + i\wt{\o})_{k} \,\wt{a_0}\,\overline{{a_0}}\ \D x
=\int_{\wt{\Omega}} p(x) \left((\o + i\wt{\o})\cdot \nabla\right)^2\left(\wt{a_0}\overline{{a_0}}\right)\ \D x.
\end{equation}
Consider $\wt{a_0}= e^{-ix\cdot \xi}$ and $a_0=(\omega\cdot x)^2$ in $\wt{\Omega}$ with $\omega\cdot\o =1$ and $\o\cdot\wt{\o}=\o\cdot\xi=0$, which  are indeed solutions of the transport equations \eqref{eq_transp_new_1} for $m>2$.
Therefore we get
\[ \int_{\wt{\Omega}} p(x)\, e^{-ix\cdot\xi}\, \D x= 0.\]
Hence, $p=0$ in $\wt{\Omega}$ and so, $\wt{V}=0$ in $\wt{\Omega}$.

This shows that  $(A_\A - \wt{A}_\A)$ (see \eqref{kh2}) is an isotropic matrix given as 
\[(A - \wt{A})_\A = d_{\#}I\quad \mbox{in }\wt{\Omega}\mbox{ for }|\A|=2.\]
Hence we fall into the regime of \textbf{Case (2)}. 

Next, we turn back to the \textbf{Case (2)} when $(A_\A-\wt{A}_\A)$ is isotropic for $|\A|=2$. In that case we have the integral identity  \eqref{iso}. 
First we show that $(A_\A-\wt{A}_\A)=0$ in $\wt{\Omega}$ for $|\A|=1$ and then we move into proving $(A_\A-\wt{A}_\A)=0$ in ${\Omega}$ for $|\A|=2$.  

Let us consider $\wt{a_0}= 1$ and  $a_0=e^{ix\cdot \xi}$  in $\wt{\Omega}$, which are indeed the solutions of the transport equations \eqref{eq_transp_new_1}. Plugging in the functions $a_0$ and $\wt{a_0}$ in \eqref{iso} we get 
\begin{equation}\label{qw8}\sum_{\lvert\A\rvert = 1}\int_{\wt{\Omega}} (A_{\A}-{\wt{A}_\A})(\o + i\wt{\o})_{\A}\,
e^{-ix\cdot\xi}\ \D x =0.
\end{equation}
This implies $(A_{\A}-{\wt{A}_\A})=\nabla \widetilde{p}$, $|\alpha|=1$, for some scalar function $\wt{p}\in W^{3,\infty}(\wt{\Omega})$.
Moreover, as $(A_\A-\wt{A}_\A)=0$ in a neighborhood of $\partial\wt{\Omega}$, hence we conclude that $\wt{p}$ is a constant ($\wt{p}|_{\partial\wt{\Omega}}=\wt{c}\in\mathbb{C})$ on $\partial\wt{\Omega}$. Hence by considering $(\wt{p}-\wt{c})$, we assume that $\wt{p}=0$ on $\partial\wt{\Omega}$. Next by putting this in \eqref{iso} we get 
\begin{equation}\label{qw7}
\int_{\wt{\Omega}} \nabla \wt{p}\cdot (\o + i\wt{\o})\, \wt{a}\, \overline{a_0}\, \D x =0.
\end{equation}
Now in \eqref{qw7} we consider $\wt{a_0}= 1$ and $a=e^{ix\cdot \xi}(\o\cdot x)$ in $\wt{\Omega}$, which are the solution of the transport equations \eqref{eq_transp_new_1}. 
Using integration by parts with the fact that $\wt{p}=0$ on $\partial\wt{\Omega}$, we obtain 
\[ \int_{\wt{\Omega}} \wt{p}(x)\, e^{-ix\cdot\xi}\, \D x= 0.\]
Therefore $\wt{p}=0$ in $\wt{\Omega}$ and hence $(A_\A-\wt{A}_\A)=0$ in ${\Omega}$, for $|\A|=1$. 

Next in order to show the same for $|\A|=2$ with $A_{\A}$ and $\wt{A}_{\A}$ being isotropic matrices, we go back to \eqref{iso} and put $A_\A=\wt{A}_\A$ for $|\A|=1$ in $\wt{\Omega}$. 
Taking $\wt{a_0}=(\o\cdot x)$ and $a_0=e^{ix\cdot\xi}$ in $\wt{\Omega}$ we get 
\[ \int_{\wt{\Omega}} m(x)\, e^{-ix\cdot\xi}\, \D x = 0, \] where $A_{\A}-\wt{A}_{\A} = m(x)I$.
Thus $m=0$ in $\wt{\Omega}$, consequently $(A_\A-\wt{A}_\A) =0 $ for $|\A|=2$ in $\wt{\Omega}$. So in the \textbf{Case(2)} we have shown $A_\A=\wt{A}_\A$ for $\A=1,2$ in $\Omega$.\\ 
\\
Now we go back to \textbf{Case (1)} again and by the above analysis we get $A_\A=\wt{A}_\A$ for $|\A|=2$ in $\wt{\Omega}$. 
Now we will show that for \textbf{Case (1)}, $A_\A=\wt{A}_\A$ for $|\A|=1$ in $\wt{\Omega}$.
In order to do that, we go back to identity \eqref{cgoIdentityE} and put $A_\A=\wt{A}_\A$ for $|\A|=2$ in $\wt{\Omega}$. 
By doing that we get first three terms in the identity becomes zero. 
Now by multiplying both sides of the identity \eqref{cgoIdentityE} by $h$ and letting $h\to 0$ we get
\[\sum_{\lvert\A\rvert = 1}\int_{\wt{\Omega}} (A_{\A}-{\wt{A}_\A})(\o + i\wt{\o})_{\A}\,\wt{a_0}
\, \overline{a_0}\, \D x =0.
\]
Here we consider $\wt{a_0}= e^{-ix\cdot\xi}$ and $a_0=1$ in $\wt{\Omega}$ and end up with having \eqref{qw8}. 
Hence we get $A_\A=\wt{A}_\A$ for $|\A|=1$ in $\wt{\Omega}$.   

Hence, in both cases we have $A_\A= \wt{A}_\A$ for $|\A|=1,2$ in $\wt{\Omega}$. 
The remaining part is to show the uniqueness of the zeroth order perturbation, that is $(q-\wt{q})=0$ in $\wt{\Omega}$. 
Putting $A_\A=\wt{A}_\A$ for $|\A|=1,2$ in $\wt{\Omega}$ in the identity \eqref{cgoIdentityE} we get the first five terms are zero. 
Next, by letting $h\to 0$ we obtain 
\[
\int_{\wt{\Omega}} (q -\widetilde{q})\,\wt{a_0}\,\overline{a_0}\, \D x =0.
\]
By considering $\wt{a_0}= e^{-ix\cdot\xi}$ and $a_0=1$ in $\wt{\Omega}$, we end up with having 
\[
\int_{\wt{\Omega}} (q -\widetilde{q})e^{-ix\cdot\xi}\, \D x =0, 
\]
which proves $q=\wt{q}$ in $\wt{\Omega}$.
This completes the  proof of the  Theorem \ref{mainresult}.
\epr

\textbf{Acknowledgment:}
{\small
The authors express their gratitude to the anonymous referees for their valuable comments and suggestions on this article.
}


\begin{thebibliography}{100}
\bibitem{AGM}
Shmuel Agmon.
\newblock {\em Lectures on elliptic boundary value problems}.
\newblock Prepared for publication by B. Frank Jones, Jr. with the assistance
  of George W. Batten, Jr. Van Nostrand Mathematical Studies, No. 2. D. Van
  Nostrand Co., Inc., Princeton, N.J.-Toronto-London, 1965.

\bibitem{Ashbaugh_Buckling}
Mark~S. Ashbaugh.
\newblock On universal inequalities for the low eigenvalues of the buckling
  problem.
\newblock In {\em Partial differential equations and inverse problems}, volume
  362 of {\em Contemp. Math.}, pages 13--31. Amer. Math. Soc., Providence, RI,
  2004.

\bibitem{AY}
Yernat~M. Assylbekov.
\newblock Inverse problems for the perturbed polyharmonic operator with
  coefficients in {S}obolev spaces with non-positive order.
\newblock {\em Inverse Problems}, 32(10):105009, 22, 2016.

\bibitem{KIY}
Yernat~M. Assylbekov and Karthik Iyer.
\newblock Determining rough first order perturbations of the polyharmonic
  operator.
\newblock {\em arXiv:1703.02569}, 2017.

\bibitem{AYY}
Yernat~M. Assylbekov and Yang Yang.
\newblock Determining the first order perturbation of a polyharmonic operator
  on admissible manifolds.
\newblock {\em J. Differential Equations}, 262(1):590--614, 2017.

\bibitem{Bhattacharyya2017}
Sombuddha Bhattacharyya.
\newblock An inverse problem for the magnetic schr\"odinger operator on
  riemannian manifolds from partial boundary data.
\newblock {\em Inverse Problems and Imaging}, 12(1930-8337-2018-3-801), 2018.

\bibitem{BUK}
A.~Bukhgeim and G.~Uhlmann.
\newblock Recovering a potential from cauchy data.
\newblock {\em Communication in Partial Differential Equations},
  27:3-4:653--688, 2007.

\bibitem{Calderon_Paper}
Alberto-P. Calder{\'o}n.
\newblock On an inverse boundary value problem.
\newblock In {\em Seminar on {N}umerical {A}nalysis and its {A}pplications to
  {C}ontinuum {P}hysics ({R}io de {J}aneiro, 1980)}, pages 65--73. Soc. Brasil.
  Mat., Rio de Janeiro, 1980.

\bibitem{APH}
Anupam~Pal Choudhury and Horst Heck.
\newblock Stability of the inverse boundary value problem for the biharmonic
  operator: {L}ogarithmic estimates.
\newblock {\em J. Inverse Ill-Posed Probl.}, 25(2):251--263, 2017.

\bibitem{VA}
Anupam~Pal Choudhury and Venkateswaran~P. Krishnan.
\newblock Stability estimates for the inverse boundary value problem for the
  biharmonic operator with bounded potentials.
\newblock {\em J. Math. Anal. Appl.}, 431(1):300--316, 2015.

\bibitem{DOS}
David Dos Santos~Ferreira, Carlos~E. Kenig, Johannes Sj{\"o}strand, and Gunther
  Uhlmann.
\newblock Determining a magnetic {S}chr\"{o}dinger operator from partial {C}auchy
  data.
\newblock {\em Comm. Math. Phys.}, 271(2):467--488, 2007.

\bibitem{Polyharmonic_Book}
Filippo Gazzola, Hans-Christoph Grunau, and Guido Sweers.
\newblock {\em Polyharmonic boundary value problems}, volume 1991 of {\em
  Lecture Notes in Mathematics}.
\newblock Springer-Verlag, Berlin, 2010.
\newblock Positivity preserving and nonlinear higher order elliptic equations
  in bounded domains.

\bibitem{VT}
T.~Ghosh and Krishnan V.
\newblock Determination of lower order perturbations of the polyharmonic
  operator from partial boundary data.
\newblock {\em Applicable Analysis: An International Journal}, 94, 2015.

\bibitem{TS}
Tuhin Ghosh.
\newblock An inverse problem on determining upto first order perturbations of a
  fourth order operator with partial boundary data.
\newblock {\em Inverse Problems}, 31(10), 2015.

\bibitem{Gerd_Grubb}
Gerd Grubb.
\newblock {\em Distributions and operators}, volume 252 of {\em Graduate Texts
  in Mathematics}.
\newblock Springer, New York, 2009.

\bibitem{IKE}
Masaru Ikehata.
\newblock A special {G}reen's function for the biharmonic operator and its
  application to an inverse boundary value problem.
\newblock {\em Comput. Math. Appl.}, 22(4-5):53--66, 1991.
\newblock Multidimensional inverse problems.

\bibitem{ISA}
Victor Isakov.
\newblock Completeness of products of solutions and some inverse problems for
  {PDE}.
\newblock {\em J. Differential Equations}, 92(2):305--316, 1991.

\bibitem{KEN}
Carlos~E. Kenig, Johannes Sj{\"o}strand, and Gunther Uhlmann.
\newblock The {C}alder\'on problem with partial data.
\newblock {\em Ann. of Math. (2)}, 165(2):567--591, 2007.

\bibitem{KRU2}
Katsiaryna Krupchyk, Matti Lassas, and Gunther Uhlmann.
\newblock Determining a first order perturbation of the biharmonic operator by
  partial boundary measurements.
\newblock {\em J. Funct. Anal.}, 262(4):1781--1801, 2012.

\bibitem{KRU1}
Katsiaryna Krupchyk, Matti Lassas, and Gunther Uhlmann.
\newblock Inverse boundary value problems for the perturbed polyharmonic
  operator.
\newblock {\em Trans. Amer. Math. Soc.}, 366(1):95--112, 2014.

\bibitem{KU}
Katsiaryna Krupchyk and Gunther Uhlmann.
\newblock Inverse boundary problems for polyharmonic operators with unbounded
  potentials.
\newblock {\em J. Spectr. Theory}, 6(1):145--183, 2016.

\bibitem{SALOT}
Mikko Salo.
\newblock Inverse problems for nonsmooth first order perturbations of the
  {L}aplacian.
\newblock {\em Ann. Acad. Sci. Fenn. Math. Diss.}, (139):67, 2004.
\newblock Dissertation, University of Helsinki, Helsinki, 2004.

\bibitem{SEROV2}
V.~S. Serov.
\newblock Borg-{L}evinson theorem for perturbations of the bi-harmonic
  operator.
\newblock {\em Inverse Problems}, 32(4):045002, 19, 2016.

\bibitem{VS}
V.~A. Sharafutdinov.
\newblock {\em Integral geometry of tensor fields}.
\newblock Inverse and Ill-posed Problems Series. VSP, Utrecht, 1994.

\bibitem{Sharafutdinov}
V.~A. Sharafutdinov.
\newblock {\em Integral geometry of tensor fields}.
\newblock VSP, Utrecht, the Netherlands, 1994.

\bibitem{SUN}
Zi~Qi Sun.
\newblock An inverse boundary value problem for {S}chr\"{o}dinger operators with
  vector potentials.
\newblock {\em Trans. Amer. Math. Soc.}, 338(2):953--969, 1993.

\bibitem{SU}
John Sylvester and Gunther Uhlmann.
\newblock A global uniqueness theorem for an inverse boundary value problem.
\newblock {\em Ann. of Math. (2)}, 125(1):153--169, 1987.

\bibitem{SEROV}
Tyni Teemu and Serov Valery.
\newblock Scattering problems for perturbations of the multidimensional
  biharmonic operator.
\newblock {\em Inverse Problems {\&} Imaging}, 12, 2018.

\bibitem{YANG}
Yang Yang.
\newblock Determining the first order perturbation of a bi-harmonic operator on
  bounded and unbounded domains from partial data.
\newblock {\em J. Differential Equations}, 257(10):3607--3639, 2014.
\end{thebibliography}
\end{document}